\documentclass{article}

\usepackage{amsmath,amscd,amsfonts,amssymb,amsthm}
\usepackage{graphicx}
\usepackage{epstopdf}
\usepackage{a4wide}
\usepackage{graphicx}
\usepackage{subfigure}
\usepackage[latin1]{inputenc}
\usepackage{color}

\newcommand{\subj}[1]{\par\noindent{\bf Mathematics Subject Classification 2010: }#1.}
\newcommand{\keyw}[1]{\par\noindent{\bf Keywords: }#1.}

\theoremstyle{definition}
\newtheorem{definition}{Definition}
\newtheorem{theorem}{Theorem}

\theoremstyle{remark}
\newtheorem{remark}{Remark}

\def\a{\alpha}
\def\r{\rho}
\def\t{\tau}

\def\LD{{^CD_{a+}^{\a,\r}}}
\def\RD{{D_{b-}^{\a,\r}}}
\def\RI{{I_{b-}^{1-\a,\r}}}

\begin{document}

\title{Variational Problems Involving a Caputo-Type Fractional Derivative}

\author{Ricardo Almeida\\
{\tt ricardo.almeida@ua.pt}}

\date{Center for Research and Development in Mathematics and Applications (CIDMA)\\
Department of Mathematics, University of Aveiro, 3810--193 Aveiro, Portugal}

\maketitle


\begin{abstract}

We study calculus of variations problems, where the Lagrange function depends on the Caputo-Katugampola fractional derivative. This type of fractional operator is a generalization of the Caputo and the Caputo--Hadamard fractional derivatives, with dependence on a real parameter $\r$. We present sufficient and necessary conditions of first and second order to determine the extremizers of a functional. The cases of integral and holomonic constraints are also considered.

\end{abstract}

\subj{26A33,34A08,34K28}

\keyw{fractional calculus, Caputo--Katugampola fractional derivative, variational problems}


\section{Introduction}

Fractional calculus of variations was first studied by Riewe in \cite{Riewe}, where he showed that nonconservative forces such as friction are modeled by non-integer order derivatives. In fact, although most known methods deal with conservative systems, natural processes are nonconservative and so the usual Lagrange formulation is inadequate to characterize these phenomena.
It turns out that fractional calculus, due to its non-local character, may better describe the behavior of the certain real processes. For this reason nowadays it is an important research area, which has attracted the attention not only of mathematicians, but also of physicists and engineers.
Generally speaking, fractional calculus deals with integrals and derivatives of arbitrary real order, and it was considered since the very beginning of calculus, but only on the past decades it has proven to be applicable in real problems. We can find several definitions for fractional operators, and to decide which one is more efficient to model the problem is a question whose answer depends on the system. Thus, we find several works that deal with similar subjects, for different kinds of fractional operators (see e.g. \cite{Almeida,Atanackovic,Baleanu10,MR2411429,MR2563910,Cresson,El-Nabulsi,Frederico:Torres10,Jumarie,Jumarie3b,comBasia:Frac1,MT,book:MT,withTatiana:Basia} and references therein). In this paper we intend to present a more general theory, that includes the Caputo and the Caputo--Hadamard fractional derivatives, following the work started in \cite{Almeida1}.

The text is organized as follows. In Section \ref{sec:FC} we review the necessary definitions and results for the Caputo--Katugampola fractional derivative.
The main results of the paper are presented in Section \ref{FCV}. First, in \S \ref{sec:fund} we consider the fundamental problem, where we present necessary and sufficient conditions that every minimizer of the functional must fulfill. In \S \ref{sec:legendre} we prove a second order condition to determine if the extremals are in fact minimizers for the functional. Then, we consider variational problems subject to an integral constraint in \S \ref{sec:iso} and to a holomonic constraint in \S \ref{sec:holo}. We end with a generalization of the variational problem, known in the literature as Herglotz problem (\S \ref{sec:herglotz}).

\section{Caputo--Katugampola fractional derivative}\label{sec:FC}

We find several definitions for fractional derivatives, each of them presenting its advantages and disadvantages.
One of those, considered mainly by engineers, is the Caputo derivative exhibiting two important features: the derivative of a constant is zero and the Laplace transform depends only on integer-order derivatives. Given a function $x:[a,b]\to\mathbb R$, the Caputo fractional derivative of $x$ of order $\a\in(0,1)$ is defined by \cite{Kilbas}
$${^CD_{a+}^\a}x(t)=\frac{1}{\Gamma(1-\a)}\frac{d}{dt}\int_a^t\frac{1}{(t-\t)^\a}[x(\t)-x(a)]\,d\t,$$
where $\Gamma$ denotes the well-known Gamma function,
$$\Gamma(z)=\int_0^\infty t^{z-1}\exp(-t)\, dt, \quad z>0,$$
and if $x$ is of class $C^1$, then we have the equivalent form
$${^CD_{a+}^\a}x(t)=\frac{1}{\Gamma(1-\alpha)}\int_a^t \frac{1}{(t-\t)^{\a}}x'(\t) \,d\tau.$$
From the definition, it is clear that if $x$ is a constant function, then ${^CD_{a+}^\a}x(t)=0$ and that, if $x(a)=0$, then the Caputo fractional derivative coincides with the Riemann--Liouville fractional derivative. The Caputo--Hadamard derivative is a very recent concept \cite{Baleanu1,Baleanu2}, and it combines
the Caputo derivative with the Hadamard fractional operator. Given a function $x$, the Caputo--Hadamard fractional derivative of order $\a$ is defined as
$${^HD_{a+}^\a}x(t)=\frac{1}{\Gamma(1-\a)}t\frac{d}{dt}\int_a^t\left(\ln\frac{t}{\t}\right)^{-\a}\frac{x(\t)-x(a)}{\t}d\t,$$
and again if  $x$ is of class $C^1$, then
$${^HD_{a+}^\a}x(t)=\frac{1}{\Gamma(1-\alpha)}\int_a^t\left(\ln\frac{t}{\t}\right)^{-\a}x'(\t)\, d\tau.$$

In \cite{Almeida1}, a new type of operator is presented, that generalizes the two previous operators, by introducing a new parameter $\r>0$ in the definition. The same idea has already been done in \cite{Katugampola1,Katugampola2}, where a new operator is defined which generalizes the Riemann--Liouville and the Hadamard fractional derivatives.

\begin{definition} Let $0<a<b<\infty$, $x:[a,b]\rightarrow\mathbb{R}$ be a function, and $\alpha\in(0,1)$ and $\rho>0$ two fixed reals.
The Caputo--Katugampola fractional derivative of order $\a$ is defined as
$$\LD x(t) =\frac{\r^\a}{\Gamma(1-\alpha)}t^{1-\r}\frac{d}{dt}\int_a^t \frac{\t^{\r-1}}{(t^\r-\t^\r)^\a}[x(\t)-x(a)] \, d\tau.$$
\end{definition}

This was motivated by the recent notion due to Katugampola in \cite{Katugampola1}, where a new fractional integral operator ${I_{a+}^{\a,\r}}x(t)$ is presented,
$${I_{a+}^{\a,\r}}x(t)=\frac{\r^{1-\a}}{\Gamma(\alpha)}\int_a^t \frac{\t^{\r-1}}{(t^\r-\t^\r)^{1-\a}}x(\t) d\tau.$$
When $\a\in\mathbb N$, the fractional integral reduces to a n-fold integral of the form
$$\int_a^t \t_1^{\r-1}\, d\t_1\,\int_a^{\t_1} \t_2^{\r-1}\, d\t_2\,\ldots \int_a^{\t_{n-1}} \t_n^{\r-1}x(\t_n)\, d\t_n.$$

If $x$ is continuously differentiable, then the fractional operator can be written in an equivalent way \cite{Almeida1}:
$$\LD x(t) =\frac{\r^\a}{\Gamma(1-\alpha)}\int_a^t \frac{1}{(t^\r-\t^\r)^\a}x'(\t)\, d\tau,$$
that is, we have
$$\LD x(t) ={I_{a+}^{1-\a,\r}}\left(\t^{1-\r}\frac{d}{d\t}x\right)(t).$$
The main results of \cite{Almeida1} are:
\begin{enumerate}
\item the operator is linear and bounded from $C([a,b])$ to $C([a,b])$,
\item if $x\in C^1[a,b]$, then the map $t\mapsto \LD x(t)$ is continuous in $[a,b]$ and $\LD x(a)=0$,
\item if $x$ is continuous, then $\LD {I_{a+}^{\a,\r}} x(t)=x(t)$,
\item if $x\in C^1[a,b]$, then ${I_{a+}^{\a,\r}}\, \LD x(t)=x(t)-x(a)$.
\end{enumerate}
We remark that, taking $\r=1$, we obtain the Caputo fractional derivative,
$${^CD_{a+}^\a}x(t)={^CD_{a+}^{\a,1}}x(t),$$
 and doing $\r\to0^+$, then we get the Caputo--Hadamard fractional derivative:
$${^HD_{a+}^\a}x(t)={^CD_{a+}^{\a,0^+}}x(t).$$

One crucial result for our present work is an integration by parts formula. For that, we need the two following auxiliary operators, a fractional integral type
$${I_{b-}^{\a,\r}} x(t)=\frac{\r^{1-\a}}{\Gamma(\alpha)}\int_t^b \frac{1}{(\t^\r-t^\r)^{1-\a}}x(\t)\,d\tau,$$
and a fractional differential type
$$\RD x(t)=\frac{\r^\a}{\Gamma(1-\alpha)}\frac{d}{dt}\int_t^b \frac{1}{(\t^\r-t^\r)^\a}x(\t)\,d\tau.$$

\begin{theorem}\label{IntegrationParts} Let $x$ be a continuous function and $y$ be a function of class $C^1$. Then, the following equality holds:
$$\int_a^b x(t) \LD y(t) \, dt=\left[y(t) \RI x(t)\right]_{t=a}^{t=b}-\int_a^b y(t)\RD x(t)  \, dt.$$
\end{theorem}

\begin{proof}
Starting with the definition, we have
$$\int_a^b x(t) \LD y(t) \, dt=\frac{\r^\a}{\Gamma(1-\alpha)}\int_a^b\int_a^t x(t) \frac{1}{(t^\r-\t^\r)^\a}y'(\t)\, d\tau\,dt.$$
Using the Dirichlet's formula, we get
$$\int_a^b\int_a^t x(t) \frac{1}{(t^\r-\t^\r)^\a}y'(\t)\, d\tau\,dt=\int_a^b\int_t^b x(\t) \frac{1}{(\t^\r-t^\r)^\a}y'(t)\, d\tau\,dt.$$
Integrating by parts, considering
$$u'(t)=y'(t) \quad \mbox{and} \quad v(t)=\int_t^b x(\t) \frac{1}{(\t^\r-t^\r)^\a}\, d\tau,$$
we obtain the desired formula.
\end{proof}

\section{The variational problem}\label{FCV}

Fractional calculus of variations appeared in 1996, with the work ok Riewe \cite{Riewe}, since as he explained, "Traditional Lagrangian and Hamiltonian mechanics cannot be used with nonconservative forces such as friction". Since then, several studies have appeared, for different types of fractional derivatives and/or fractional integrals, namely to  determine necessary and sufficient conditions that any extremal for the variational functional must satisfy.

To start, we present the concept of minimizer for a given functional. On the space $C^1[a,b]$, consider the norm $\|\cdot\|$ given by
$$\|x\| =\max_{t\in[a,b]}|x(t)|+ \max_{t\in[a,b]}\left|\LD x(t)\right|.$$
Let $D\subseteq C^1[a,b]$ be a nonempty set and $J$ a functional defined on $D$. We say that $x$ is a local minimizer of $J$ in the set $D$ if there exists a
neighborhood $V_\delta(x)$ such that for all $x^* \in V_\delta(x)\cap D$, we have $J(x)\leq J(x^*)$. Note that any function  $x^*\in V_\delta(x)\cap D$ can be represent in the form $x^*=x+\epsilon h$, where $|\epsilon| \ll 1$ and $h$ is such that $x+\epsilon h\in D$.

The purpose of this work is to study fractional calculus of variations problems, where the integral functional depends on the Caputo--Katugampola fractional derivative. Given $x\in C^1[a,b]$, consider the functional
\begin{equation}\label{functional}J(x)=\int_a^b L(t,x(t),\LD x(t))\,dt,\end{equation}
with the following assumptions:
\begin{enumerate}
\item $L:[a,b]\times \mathbb R^2\to\mathbb R$ is continuously differentiable with respect to the second and third arguments;
\item given any function $x$, the map $t\mapsto \RD(\partial_3L (t,x(t),\LD x(t)))$ is continuous.
\end{enumerate}
Here, and along the work, given a function with several independent variables $f:A\subseteq \mathbb R^n\to\mathbb R$, we denote
$$\partial_i f(x_1,\ldots,x_n):=\frac{\partial f}{\partial x_i}(x_1,\ldots,x_n).$$

We remark that $x$ may be fixed or free at $t=a$ and $t=b$. Both cases will be considered later.

\subsection{The fundamental problem}\label{sec:fund}

We seek necessary and sufficient conditions that each extremizers of the functional must fulfill. In order to obtain such equations we consider variations of the solutions and use the fact that the first variation of the functional must vanish at the minimizer.

\begin{theorem} Let $x$ be a minimizer of the functional $J$ as in \eqref{functional}, defined on the subspace
$$U=\left\{ x\in C^1[a,b]\, : \, x(a)=x_a \quad \mbox{and} \quad x(b)=x_b \right\},$$
where $x_a,x_b\in\mathbb R$ are fixed. Then, $x$ is a solution for the fractional differential equation
\begin{equation}\label{ELequation}\partial_2L(t,x(t),\LD x(t))-\RD \left(\partial_3L(t,x(t),\LD x(t))\right)=0\end{equation}
on $[a,b]$.
\end{theorem}
\begin{proof}
Let $x+\epsilon h$ be a variation of $x$, with $|\epsilon| \ll 1$ and $h\in C^1[a,b]$. Since $x+\epsilon h$ must belong to the set $U$, the boundary conditions $h(a)=0=h(b)$ must hold. Define the function $j$ in a neighborhood of zero as
$$j(\epsilon)=J(x+\epsilon h).$$
Since $x$ is a minimizer of $J$, then $\epsilon=0$ is a minimizer of $j$ and so $j'(0)=0$. Computing $j'(0)$ and using Theorem \ref{IntegrationParts}, we get
$$\int_a^b\left[\partial_2L(t,x(t),\LD x(t))-\RD \left(\partial_3L(t,x(t),\LD x(t))\right)\right] h(t)\,dt$$
\begin{equation}\label{aux1}+\left[h(t) \RI \left(\partial_3L(t,x(t),\LD x(t))\right)\right]_{t=a}^{t=b}=0.\end{equation}
Since $h(a)=0=h(b)$ and $h$ is arbitrary elsewhere, we conclude that
$$\partial_2L(t,x(t),\LD x(t))-\RD \left(\partial_3L(t,x(t),\LD x(t))\right)=0, \quad \forall t\in[a,b].$$
\end{proof}

\begin{definition} Eq. \eqref{ELequation} is called the Euler--Lagrange equation associated to functional $J$. The solutions of this fractional differential equation are called extremals of $J$.
\end{definition}

\begin{remark} The case of several dependent variables is similar, and the Euler--Lagrange equations are easily deduced. Let
$$J(\overline x)=\int_a^b L(t,\overline x(t),\LD \overline x(t))\,dt,$$
where
\begin{enumerate}
\item $\overline x(t)=(x_1(t),\ldots,x_n(t))$ and  $\LD\overline x(t)=(\LD x_1(t),\ldots,\LD x_n(t))$;
\item $L:[a,b]\times \mathbb R^{2n}\to\mathbb R$ is continuously differentiable with respect to its $i$th argument, for all $i\in\{2,\ldots, 2n+1\}$;
\item given any function $x$, the maps $t\mapsto \RD(\partial_{n+i}L (t,x(t),\LD x(t)))$ is continuous, for all $i\in\{2,\ldots, n+1\}$.
\end{enumerate}
In this case, if $x$ is a minimizer of the functional $J$, subject to the restrictions $\overline x(a)= \overline x_a$ and $\overline x(b)= \overline x_b$,
where $ \overline x_a,\overline x_b\in\mathbb R^n$ are fixed, then $x$ is a solution of
$$\partial_{i}L(t,x(t),\LD x(t))-\RD \left(\partial_{n+i}L(t,x(t),\LD x(t))\right)=0,$$
for all $i\in\{2,\ldots, n+1\}$.
\end{remark}

We remark that Eq. \eqref{ELequation} gives only a necessary condition. To deduce a sufficient condition, we recall the notion of convex function.
Given a function $L(\underline t,x,y)$  continuously differentiable with respect to the second and third arguments, we say that $L$ is convex in $S\subseteq\mathbb R^3$ if condition
$$L(t,x+x_1,y+y_1)-L(t,x,y)\geq\partial_2 L(t,x,y)x_1+\partial_3 L(t,x,y)y_1$$
holds for all $(t,x,y),(t,x+x_1,y+y_1)\in S$.

\begin{theorem}\label{Teo:suffi} If the function $L$ as in \eqref{functional} is convex in $[a,b]\times\mathbb R^2$, then each solution of the fractional Euler--Lagrange equation \eqref{ELequation} minimizes $J$ in $U$.
\end{theorem}

\begin{proof} Let $x$ be a solution of Eq. \eqref{ELequation}  and $x+\epsilon h$ be a variation of $x$, with $|\epsilon| \ll 1$ and $h\in C^1[a,b]$ with $h(a)=0=h(b)$. Then
$$\begin{array}{ll}
J(x+\epsilon h)-J(x)&= \displaystyle\int_a^b \left[L(t,x(t)+\epsilon h(t),\LD x(t)+\epsilon \LD h(t))-L(t,x(t),\LD x(t))\right] \, dt \\
& \geq \displaystyle \int_a^b \left[ \partial_2 L(t,x(t),\LD x(t))\epsilon h(t)+ \partial_3 L(t,x(t),\LD x(t))\epsilon \LD h(h) \right] \, dt\\
&=  \displaystyle\int_a^b \left[ \partial_2 L(t,x(t),\LD x(t)) -\RD ( \partial_3 L(t,x(t),\LD x(t)))\right] \epsilon h(t)\,dt\\
&=0
\end{array}$$
since $x$ is a solution of (\ref{ELequation}). Therefore, $x$ is a local minimizer of $J$.
\end{proof}

If we do not impose any restrictions on the boundaries, we obtain two transversality conditions.

\begin{theorem} Let $x$ be a minimizer of the functional $J$ as in \eqref{functional}. Then, $x$ is a solution for the fractional differential equation
$$\partial_2L(t,x(t),\LD x(t))-\RD \left(\partial_3L(t,x(t),\LD x(t))\right)=0, \quad t\in[a,b].$$
If $x(a)$ is free, then
$$\RI \left(\partial_3L(t,x(t),\LD x(t))\right) \quad \mbox{at} \quad t=a.$$
If $x(b)$ is free, then
$$\RI \left(\partial_3L(t,x(t),\LD x(t))\right) \quad \mbox{at} \quad t=b.$$
\end{theorem}
\begin{proof}
Since $x$ is a minimizer, then
$$\partial_2L(t,x(t),\LD x(t))-\RD \left(\partial_3L(t,x(t),\LD x(t))\right)=0,$$
for all $t\in[a,b]$. Therefore, using Eq. \eqref{aux1}, we obtain
$$\left[h(t) \RI \left(\partial_3L(t,x(t),\LD x(t))\right)\right]_{t=a}^{t=b}=0.$$
If $x(a)$ is free, then $h(a)$ is also free and taking $h(a)\not=0$ and $h(b)=0$, we get
$$\RI \left(\partial_3L(t,x(t),\LD x(t))\right) \quad \mbox{at} \quad t=a.$$
The second case is similar.
\end{proof}

Observe that in the previous results, the initial point of the cost functional coincides with the initial point of the fractional derivative. Next we consider a more general type of problems, by considering $A\in(a,b)$ and the functional
\begin{equation}\label{funcA}J(x)=\int_A^b L(t,x(t),\LD x(t))\,dt,\end{equation}
with the same assumptions on $L$ as previously, defined on the set
$$U_A=\left\{ x\in C^1[a,b]\, : \, x(A)=x_A \quad \mbox{and} \quad x(b)=x_b \right\}.$$

\begin{theorem} If $x$ is a minimizer of the functional $J$ as in \eqref{funcA}, then $x$ satisfies the fractional differential equations
$${D_{A-}^{\a,\r}}\left(\partial_3L(t,x(t),\LD x(t))\right)-\RD \left(\partial_3L(t,x(t),\LD x(t))\right)=0 \quad \mbox{on} \, [a,A],$$
$$\partial_2L(t,x(t),\LD x(t))-\RD \left(\partial_3L(t,x(t),\LD x(t))\right)=0 \quad \mbox{on} \, [A,b],$$
and the transversality condition
$${I_{A-}^{1-\a,\r}}\left(\partial_3L(t,x(t),\LD x(t))\right)-\RI \left(\partial_3L(t,x(t),\LD x(t))\right)=0 \quad \mbox{at} \quad t=a.$$
\end{theorem}

\begin{proof}
Let $x+\epsilon h$ be a variation of $x$, with $|\epsilon| \ll 1$, and $h\in C^1[a,b]$ with $h(A)=0=h(b)$. Computing the first variation of $J$, we deduce the following
$$\begin{array}{ll}
0&=\displaystyle \int_A^b\left[\partial_2L(t,x(t),\LD x(t))h(t)+\partial_3L(t,x(t),\LD x(t))\LD h(t)\right]\,dt\\
&=\displaystyle \int_a^b\left[\partial_2L(t,x(t),\LD x(t))h(t)+\partial_3L(t,x(t),\LD x(t))\LD h(t)\right]\,dt\\
&\displaystyle\quad -\int_a^A\left[\partial_2L(t,x(t),\LD x(t))h(t)+\partial_3L(t,x(t),\LD x(t))\LD h(t)\right]\,dt.
 \end{array}$$
Integrating by parts, and using the fact that $h(A)=0=h(b)$, we arrive at
$$\int_a^A\left[{D_{A-}^{\a,\r}}\left(\partial_3L(t,x(t),\LD x(t))\right)-\RD \left(\partial_3L(t,x(t),\LD x(t))\right)\right] h(t)\,dt$$
$$+ \int_A^b\left[\partial_2L(t,x(t),\LD x(t))-\RD \left(\partial_3L(t,x(t),\LD x(t))\right)\right] h(t)\,dt$$
$$+h(a)\left[{I_{A-}^{1-\a,\r}}\left(\partial_3L(t,x(t),\LD x(t))\right)-\RI \left(\partial_3L(t,x(t),\LD x(t))\right)\right]^{t=a}=0.$$
Since $h$ is an arbitrary function, we obtain the three necessary conditions.
\end{proof}

One interesting question is to determine the best type of fractional derivative for which the functional attains the minimum possible value. The Caputo--Katugampola fractional derivative depends on an extra parameter $\r$, and we can obtain e.g. the Caputo and the Caputo--Hadamard fractional derivatives when $\r=1$ and $\r\to0^+$, respectively. Thus, we are interested now to determine not only the minimizer $x$ but also the value of $\r$ for which $J$ attains its minimum value.

\begin{theorem} Let $(x,\r)$ be a minimizer of the functional $J_\r$ given by
$$J_\r(x,\r)=\int_a^b L(t,x(t),\LD x(t))\,dt,$$
defined on the subspace $U\times\mathbb R^+$, where
$$U=\left\{ x\in C^1[a,b]\, : \, x(a)=x_a \quad \mbox{and} \quad x(b)=x_b \right\}.$$
Then, $(x,\r)$ is a solution for the fractional differential equation
$$\partial_2L(t,x(t),\LD x(t))-\RD \left(\partial_3L(t,x(t),\LD x(t))\right)=0$$
on $[a,b]$, and satisfies the integral equation
$$\int_a^b \partial_3L(t,x(t),\LD x(t)) \psi'(\r)\,dt=0,$$
where $\psi(\r)=\LD x(t)$.
\end{theorem}
\begin{proof}
Let $(x+\epsilon h,\r+\epsilon\r_0)$ be a variation of $(x,\r)$, with $|\epsilon| \ll 1$, $h\in C^1[a,b]$ with  $h(a)=0=h(b)$ and $\r_0\in\mathbb R$. If we define $j$ as
$$j(\epsilon)=J(x+\epsilon h,\r+\epsilon\r_0),$$
then
$$0=j'(0)=\int_a^b\left[\partial_2L(t,x(t),\LD x(t))-\RD \left(\partial_3L(t,x(t),\LD x(t))\right)\right] h(t)\,dt$$
$$+\int_a^b \partial_3L(t,x(t),\LD x(t)) \psi'(\r)\r_0\,dt.$$
If we consider $\r_0=0$, by the arbitrariness of $h$ on $(a,b)$, we conclude that
$$\partial_2L(t,x(t),\LD x(t))-\RD \left(\partial_3L(t,x(t),\LD x(t))\right)=0, \quad \forall t\in[a,b].$$
So,
$$\int_a^b \partial_3L(t,x(t),\LD x(t)) \psi'(\r)\r_0\,dt=0.$$
Taking $\r_0=1$, we get the second necessary condition.
\end{proof}

\subsection{The Legendre condition}\label{sec:legendre}

The Legendre condition is a second-order condition which an extremal must fulfill in order to be a minimizer of the functional. We suggest  \cite{Lazo} where a similar problem is solved for functionals depending on the Riemann--Liouville fractional derivative. To simplify notation, we introduce the following
$$\partial^2_{ij} f(x_1,\ldots,x_n):=\frac{\partial^2f}{\partial x_i\partial x_j}(x_1,\ldots,x_n).$$

\begin{theorem} Let $x$ be a minimizer of the functional $J$ as in \eqref{functional}, defined on the subspace $U$. If $\partial^2_{ij}L $ exists and is continuous for $i,j\in\{2,3\}$, then $x$ satisfies
\begin{equation}\label{2nd_equation}\partial^2_{33}L(t,x(t),\LD x(t))\geq0\end{equation}
on $[a,b]$.
\end{theorem}
\begin{proof}
Let $x+\epsilon h$ be a variation of $x$, with $h\in C^1[a,b]$ such that $h(a)=0=h(b)$. If we define
$$j(\epsilon)=J(x+\epsilon h),$$
then we must have $j''(\epsilon)\geq0$, that is
$$\int_a^b\left[\partial^2_{22}L(t,x(t),\LD x(t))h^2(t)+2\partial^2_{23}L(t,x(t),\LD x(t))h(t)\LD h(t)\right.$$
\begin{equation}\label{2nd}\left.+\partial^2_{33}L(t,x(t),\LD x(t))(\LD h(t))^2\right] \,dt\geq0.\end{equation}
Assume that there exists some $t_0 \in[a,b]$ for which
$$\partial^2_{33}L(t_0,x(t_0),\LD x(t_0))<0.$$
Thus, there exists a subinterval $[c,d]\subseteq[a,b]$ and three real constants $C_1,C_2,C_3$ with $C_3<0$ such that
\begin{equation}\label{2nd_C}\partial^2_{22}L(t,x(t),\LD x(t))<C_1, \quad \partial^2_{23}L(t,x(t),\LD x(t))<C_2, \quad \partial^2_{33}L(t,x(t),\LD x(t))<C_3,\end{equation}
for all $t\in[c,d]$.
Define the function $\underline h:[c,d]\to\mathbb R$ by
$$ \underline h(t)=(\a+2)(t^\r-c^\r)^{1+\a}-2\frac{\a+4}{d^\r-c^\r}(t^\r-c^\r)^{2+\a}+\frac{\a+10}{(d^\r-c^\r)^2}(t^\r-c^\r)^{3+\a}
-\frac{4}{(d^\r-c^\r)^3}(t^\r-c^\r)^{4+\a}.$$
Then, $\underline h$ is of class $C^1$, $\underline h(c)=0=\underline h(d)$,  $\underline{h}'(c)=0=\underline h'(d)$ and ${^CD_{c+}^{\a,\r}}\underline h(c)=0={^CD_{c+}^{\a,\r}}\underline h(d)$. Also, we have the following upper bounds:
$$\underline h(t)\leq 28(d^\r-c^\r)^{1+\a} \quad \mbox{and} \quad {^CD_{c+}^{\a,\r}}\underline h(t)\leq C(d^\r-c^\r), \, C=50 \r^\a\Gamma(2+\a),$$
for all $t\in[c,d]$.
Thus, the function $h:[a,b]\to\mathbb R$ defined by
$$h(t)=\left\{\begin{array}{ll}
\underline h(t) & \quad \mbox{ if } \, t\in[c,d]\\
0 & \quad \mbox{otherwise}
\end{array}\right.$$
is of class $C^1$, $h(a)=0=h(b)$ and its fractional derivative
$$\LD h(t)=\left\{\begin{array}{ll}
{^CD_{c+}^{\a,\r}} \underline h(t) & \quad \mbox{ if } \, t\in[c,d]\\
0 & \quad \mbox{otherwise}
\end{array}\right.$$
is continuous. Inserting this variation $h$ into the second order condition \eqref{2nd}, and using relations \eqref{2nd_C}, we obtain
$$\int_a^b\left[\partial^2_{22}L(t,x(t),\LD x(t))h^2(t)+2\partial^2_{23}L(t,x(t),\LD x(t))h(t)\LD h(t)\right.$$
$$\left.+\partial^2_{33}L(t,x(t),\LD x(t))(\LD h(t))^2\right] \,dt$$
$$=\int_c^d\left[\partial^2_{22}L(t,x(t),\LD x(t))h^2(t)+2\partial^2_{23}L(t,x(t),\LD x(t))h(t)\LD h(t)\right.$$
$$\left.+\partial^2_{33}L(t,x(t),\LD x(t))(\LD h(t))^2\right] \,dt$$
$$\leq \int_c^d \left[C_1 28^2(d^\r-c^\r)^{2+2\a}+C_2 56 C(d^\r-c^\r)^{2+\a}+C_3 C^2(d^\r-c^\r)^{2}\right]\,dt$$
$$=(d^\r-c^\r)^{2}(d-c)\left[C_1 28^2(d^\r-c^\r)^{2\a}+C_2 56 C(d^\r-c^\r)^{\a}+C_3 C^2\right],$$
which is negative if $d^\r-c^\r$ is chosen arbitrarily small, and thus we obtain a contradiction.
\end{proof}

\subsection{The isoperimetric problem}\label{sec:iso}

The isoperimetric problem is an old question, and was considered first by the ancient Greeks. It is related to finding a closed plane curve with a fixed perimeter $l$ which encloses the greatest area, that is, if we consider the place curve with parametric equations $(x(t),y(t))$,  $t \in [a,b]$, then we wish to maximize the functional
$$J(x,y)=\frac12\int_a^b (x(t)y'(t)-x'(t)y(t))\,dt,$$
under the boundary restrictions
$$x(a)=x(b) \quad \mbox{and} \quad y(a)=y(b),$$
and the isoperimetric constraint
$$\int_a^b\sqrt{(x')^2(t)+(y')^2(t)}\,dt=l.$$
Only in 1841, a rigorous proof that the solution is a circle was obtained by Steiner.

Nowadays, any variational problem that involves an integral constraint is called an isoperimetric problem. For the following, let $l\in\mathbb R$ be fixed,
$g:[a,b]\times \mathbb R^2\to\mathbb R$ be continuously differentiable with respect to the second and third arguments such that, for any function $x\in C^1[a,b]$, the map $t\mapsto \RD(\partial_3g (t,x(t),\LD x(t)))$ is continuous.
The integral constraint that we will consider is the following:
\begin{equation}\label{isoperi}I(x)=\int_a^b g(t,x(t),\LD x(t))\,dt=l.\end{equation}

\begin{theorem} Let $x$ be a minimizer of the functional $J$ as in \eqref{functional}, defined on the subspace
$$U=\left\{ x\in C^1[a,b]\, : \, x(a)=x_a \quad \mbox{and} \quad x(b)=x_b \right\},$$
subject to the additional restriction \eqref{isoperi}. If $x$ is not an extremal of $I$, then there exists a real $\lambda$ such that, defining the function  $K:[a,b]\times \mathbb R^2\to\mathbb R$ by $K=L+\lambda g$,  $x$ is a solution of the equation
\begin{equation}\label{eq:EL_ISO}\partial_2K(t,x(t),\LD x(t))-\RD \left(\partial_3K(t,x(t),\LD x(t))\right)=0\end{equation}
on $[a,b]$.
\end{theorem}
\begin{proof}
Consider a variation  of $x$ with two parameters $x+\epsilon_1 h_1+\epsilon_2 h_2$, with $|\epsilon_i| \ll 1$ and $h_i\in C^1[a,b]$ satisfying $h_i(a)=0=h_i(b)$, for $i=1,2$. Define the functions $i$ and $j$ with two parameters $(\epsilon_1,\epsilon_2)$ in a neighborhood of zero as
$$i(\epsilon_1,\epsilon_2)=I(x+\epsilon_1 h_1+\epsilon_2 h_2)-l,$$
and
$$j(\epsilon_1,\epsilon_2)=L(x+\epsilon_1 h_1+\epsilon_2 h_2).$$
Using Theorem \ref{IntegrationParts}, we get
$$\frac{\partial i}{\partial \epsilon_2}(0,0)=\int_a^b\left[\partial_2g(t,x(t),\LD x(t))-\RD \left(\partial_3g(t,x(t),\LD x(t))\right)\right] h_2(t)\,dt$$
$$+\left[h_2(t) \RI \left(\partial_3g(t,x(t),\LD x(t))\right)\right]_{t=a}^{t=b}=0.$$
Since $h_2(a)=0=h_2(b)$ and $x$ is not an extremal of $I$, there exists some function $h_2$ such that
$$\frac{\partial i}{\partial \epsilon_2}(0,0)\not=0.$$
Thus, by the Implicit Function Theorem, there exists an unique function $\epsilon_2(\cdot)$ defined in a neighborhood of zero such that $i(\epsilon_1,\epsilon_2(\epsilon_1))=0$, that is, there exists a subfamily of variations that satisfy the isoperimetric constraint \eqref{isoperi}.

On the other hand, $(0,0)$ is a minimizer of $j$, under the restriction $i(\cdot, \cdot)=0$, and we just proved that $\nabla i(0,0)\not=0$. Appealing to the Lagrange Multiplier Rule, there exists a real $\lambda$ such that $\nabla(j+\lambda i)(0,0)=0$. Differentiating the map
$$\epsilon_1\mapsto j(\epsilon_1,\epsilon_2)+\lambda i(\epsilon_1,\epsilon_2),$$
and putting $(\epsilon_1,\epsilon_2)=(0,0)$, we get
$$\int_a^b\left[\partial_2K(t,x(t),\LD x(t))-\RD \left(\partial_3K(t,x(t),\LD x(t))\right)\right] h_1(t)\,dt$$
$$+\left[h_1(t) \RI \left(\partial_3K(t,x(t),\LD x(t))\right)\right]_{t=a}^{t=b}=0.$$
Using the boundary conditions $h_1(a)=0=h_1(b)$, we prove the desired.
\end{proof}

\begin{remark} We can include the case where $x$ is an extremal of $I$. In this case, we apply the general form of the Lagrange Multiplier Rule, that is, there exist two reals $\lambda_0$ and $\lambda$, not both zeros, such that if we define the function  $K_0:[a,b]\times \mathbb R^2\to\mathbb R$ by $K_0=\lambda_0L+\lambda g$,  $x$ is a solution of the equation
$$\partial_2K_0(t,x(t),\LD x(t))-\RD \left(\partial_3K_0(t,x(t),\LD x(t))\right)=0$$
on $[a,b]$.
\end{remark}

\begin{theorem}\label{teo:suffi_ISO4} Suppose that the functions $L$ and $g$ as in \eqref{functional} and \eqref{isoperi} are convex in $[a,b]\times\mathbb R^2$, and let $\lambda\geq0$ be a real. Define the function $K=L+\lambda g$. Then, each solution $x$ of the fractional Euler--Lagrange equation \eqref{eq:EL_ISO}
minimizes $J$ in $U$, subject to the integral constraint \eqref{isoperi}.
\end{theorem}

\begin{proof} First, observe that the function $K$ is convex. So, by Theorem \ref{Teo:suffi}, we conclude that $x$ minimizes $K$, that is, for all variations $x+\epsilon h$, we have
$$\int_a^b L(t,x(t)+\epsilon h(t),\LD x(t)+\epsilon \LD h(t))\, dt+\int_a^b \lambda g(t,x(t)+\epsilon h(t),\LD x(t)+\epsilon \LD h(h)) \, dt$$
$$\geq \int_a^b L(t,x(t),\LD x(t))\, dt+\int_a^b \lambda g(t,x(t),\LD x(t)) \, dt.$$
Using the integral constraint, we obtain
$$\int_a^b L(t,x(t)+\epsilon h(t),\LD x(t)+\epsilon \LD h(t))\, dt+l\geq \int_a^b L(t,x(t),\LD x(t))\, dt+l,$$
and so
$$\int_a^b L(t,x(t)+\epsilon h(t),\LD x(t)+\epsilon \LD h(t))\, dt\geq \int_a^b L(t,x(t),\LD x(t))\, dt,$$
proving the desired.
\end{proof}

\begin{remark} Theorem \ref{teo:suffi_ISO4} remains valid under the assumptions that  $L$ is convex in $[a,b]\times\mathbb R^2$, $\lambda\leq0$ and $g$ is concave in $[a,b]\times\mathbb R^2$, that is,
$$g(t,x+x_1,y+y_1)-g(t,x,y)\leq\partial_2 g(t,x,y)x_1+\partial_3 g(t,x,y)y_1$$
holds for all $(t,x,y),(t,x+x_1,y+y_1)\in [a,b]\times\mathbb R^2$.
\end{remark}

\subsection{Holonomic constraints}\label{sec:holo}

Consider the functional $J$ defined by
\begin{equation}\label{funct2}J(x_1,x_2)=\int_a^b L(t,x_1(t),x_2(t),\LD x_1(t),\LD x_2(t))\,dt,\end{equation}
on the space
$$U=\left\{ (x_1,x_2)\in C^1[a,b]\times C^1[a,b]\, : \, (x_1(a),x_2(a))=x_a \quad \mbox{and} \quad (x_1(b),x_2(b))=x_b \right\},$$
with $x_a,x_b\in\mathbb R^2$ fixed. We are assuming that the Lagrangian verifies the two following conditions
\begin{enumerate}
\item $L:[a,b]\times \mathbb R^4\to\mathbb R$ is continuously differentiable with respect to its $i$th argument, for $i=2,3,4,5$;
\item given any function $x$, the maps $t\mapsto \RD(\partial_iL (t,x(t),\LD x(t)))$ are continuous, $i=4,5$.
\end{enumerate}
We consider in this new variational problem an extra constraint (called in the literature as holomonic constraint). Assume that the admissible functions lie on the surface
\begin{equation}\label{rest2}g(t,x_1(t),x_2(t))=0,\end{equation}
where $g:[a,b]\times \mathbb R^2\to\mathbb R$ is continuously differentiable with respect to the second and third arguments. For simplicity, we denote
$$x=(x_1,x_2) \quad \mbox{and} \quad \LD x=(\LD x_1,\LD x_2).$$

\begin{theorem}  Let $x\in U$ be a minimizer of $J$ as in \eqref{funct2}, under the restriction \eqref{rest2}. If
$$\partial g_3(t,x(t))\not=0 \quad \forall t\in[a,b],$$
then there exists a continuous function $\lambda:[a,b]\to\mathbb R$ such that $x$ is a solution of the two next equations
\begin{equation}\label{eq:EL:holomo}\partial_2L(t,x(t),\LD x(t))-\RD \left(\partial_4L(t,x(t),\LD x(t))\right)+\lambda(t) \partial _2 g(t,x(t))=0\end{equation}
and
$$\partial_3L(t,x(t),\LD x(t))-\RD \left(\partial_5L(t,x(t),\LD x(t))\right)+\lambda(t) \partial _3 g(t,x(t))=0$$
on $[a,b]$.
\end{theorem}
\begin{proof}
Consider a variation  of $x$ of type $x+\epsilon h$, with $|\epsilon| \ll 1$, and $h\in C^1[a,b]\times C^1[a,b]$ satisfying the boundary conditions $h(a)=(0,0)=h(b)$.
Since
$$\partial g_3(t,x(t))\not=0, \quad \forall t \in[a,b],$$
by the Implicit Function Theorem, there exists a subfamily of variations that satisfy the restriction \eqref{rest2}, that is, there exists an unique function $h_2(\epsilon,h_1)$ such that $(x_1(t)+\epsilon h_1(t),x_2(t)+\epsilon h_2(t))$ satisfies \eqref{rest2}.
So, since for all $t\in[a,b]$, we have
$$g(t,x_1(t)+\epsilon h_1(t),x_2(t)+\epsilon h_2(t))=0,$$
differentiating with respect to $\epsilon$ and putting $\epsilon=0$, we get
\begin{equation}\label{aux2} \partial g_2(t,x(t))h_1(t)+\partial g_3(t,x(t))h_2(t)=0 .\end{equation}
Define the function
\begin{equation}\label{aux3}\lambda(t)=-\frac{\partial_3L(t,x(t),\LD x(t))-\RD \left(\partial_5L(t,x(t),\LD x(t))\right)}{\partial g_3(t,x(t))}.\end{equation}
Using equations \eqref{aux2} and \eqref{aux3}, we obtain
\begin{equation}\label{aux4}\lambda(t)\partial g_2(t,x(t))h_1(t)=\left[\partial_3L(t,x(t),\LD x(t))-\RD \left(\partial_5L(t,x(t),\LD x(t))\right)\right]h_2(t).\end{equation}
On the other hand, since $x$ is a minimizer of $J$, the first variation of $J$ must vanish:
$$\int_a^b \partial_2L(t,x(t),\LD x(t))h_1(t)+\partial_3L(t,x(t),\LD x(t))h_2(t)$$
$$+\partial_4L(t,x(t),\LD x(t))\LD h_1(t)+\partial_5L(t,x(t),\LD x(t))\LD h_2(t)\,dt=0.$$
Integrating by parts, we obtain
$$\int_a^b\left[\partial_2L(t,x(t),\LD x(t))-\RD \left(\partial_4L(t,x(t),\LD x(t))\right)\right] h_1(t)$$
$$+\left[\partial_3L(t,x(t),\LD x(t))-\RD \left(\partial_5L(t,x(t),\LD x(t))\right)\right] h_2(t)\,dt=0.$$
Inserting Eq. \eqref{aux4} into the integral, we get
$$\int_a^b\left[\partial_2L(t,x(t),\LD x(t))-\RD \left(\partial_4L(t,x(t),\LD x(t))\right)+\lambda(t)\partial g_2(t,x(t))\right] h_1(t)\,dt=0,$$
and since $h_1$ is arbitrary, we have that for all $t\in[a,b]$, $x$ is a solution of the equation
$$\partial_2L(t,x(t),\LD x(t))-\RD \left(\partial_4L(t,x(t),\LD x(t))\right)+\lambda(t)\partial g_2(t,x(t))=0.$$
Also, using Eq. \eqref{aux3}, we obtain the second condition
$$\partial_3L(t,x(t),\LD x(t))-\RD \left(\partial_5L(t,x(t),\LD x(t))\right)+\lambda(t)\partial g_3(t,x(t))=0.$$
\end{proof}

\begin{theorem} Suppose that the function $L(\underline t,x_1,x_2,y_1,y_2)$ as in \eqref{funct2} is convex in $[a,b]\times\mathbb R^4$, $g:[a,b]\times \mathbb R^2\to\mathbb R$ is continuously differentiable  with respect to the second and third arguments with
$$\partial g_3(t,x(t))\not=0, \quad \forall t \in[a,b],$$
and let $\lambda$ be given by Eq. \eqref{aux3}. If $x$ is a solution of the fractional Euler--Lagrange equation \eqref{eq:EL:holomo},
then $x$ minimizes $J$ in $U$, subject to the constraint \eqref{rest2}.
\end{theorem}

\begin{proof} If $x+\epsilon h$ is a variation of $x$, then
$$\begin{array}{ll}
\displaystyle J(x+\epsilon h)-J(x)&\geq \displaystyle\int_a^b \left[ \partial_2 L(t,x(t),\LD x(t)) -\RD ( \partial_4 L(t,x(t),\LD x(t)))\right] \epsilon h_1(t)\\
\displaystyle&\quad+\left[ \partial_3 L(t,x(t),\LD x(t)) -\RD ( \partial_5 L(t,x(t),\LD x(t)))\right] \epsilon h_2(t)\,dt.\end{array}$$
Since the variation functions must satisfy the constraint \eqref{rest2}, we have the following relation
$$h_2(t)=-\frac{\partial g_2(t,x(t))h_1(t)}{\partial g_3(t,x(t))},$$
and using Eq. \eqref{aux3}, we deduce
$$\begin{array}{ll}
\displaystyle J(x+\epsilon h)-J(x)&\geq \displaystyle\int_a^b \left[\partial_2L(t,x(t),\LD x(t))-\RD \left(\partial_4L(t,x(t),\LD x(t))\right)\right.\\
\displaystyle&\quad\left.+\lambda(t) \partial _2 g(t,x(t))\right] \epsilon h_1(t)\, dt,\end{array}$$
which is zero by hypothesis.
\end{proof}

\subsection{The Herglotz problem}\label{sec:herglotz}

The fractional Herglotz problem is described in the following way. Determine a curve $x\in C^1[a,b]$ subject to the boundary conditions $x(a)=x_a$ and $x(b)=x_b$, such that the solution $z$ of the system
\begin{equation}\label{Herglotz}
\left\{\begin{array}{l}
z'(t)=L(t,x(t),\LD x(t),z(t)), \quad  t\in[a,b]\\
z(a)=z_a\\
\end{array}\right.\end{equation}
attains a minimum at $t=b$.
The assumptions are
\begin{enumerate}
\item  given any function $x$,  the map $t\mapsto {^C_aD^\alpha_t}x(t)$  is continuously differentiable;
\item $L:[a,b]\times \mathbb R^3\to\mathbb R$ is continuously differentiable with respect to its $i$th argument, for $i=2,3,4$;
\item given any function $x$, the map $t\mapsto \RD(\lambda(t)\partial_3L (t,x(t),\LD x(t)),z(t))$ is continuous, where
\begin{equation}\label{lambda}\lambda(t)=\exp\left(-\int_a^t \partial _4L(\t,x(\t),\LD x(\t),z(\t))d\tau\right).\end{equation}
\end{enumerate}

We note that, given $x$, system \eqref{Herglotz} becomes an initial value problem
$$\left\{\begin{array}{l}
z'(t)=f(t,z(t)), \quad  t\in[a,b]\\
z(a)=z_a\\
\end{array}\right.$$
and so the solution depends on $t$ and on $x$, that is,  $z=z[t,x]$. If we consider variations of $x$ of type $x+\epsilon h$ into Eq. \eqref{Herglotz}, then the solution also depends on $\epsilon$, and it is differentiable with respect to $\epsilon$ (see Section 2.6 in \cite{Anosov}).

\begin{theorem} Let $x$ be such that $z(b)$ as in Eq. \eqref{Herglotz} attains a minimum. Then, $x$ is a solution of the fractional differential equation
$$\lambda(t)\partial_2L(t,x(t),\LD x(t),z(t))-\RD\left(\lambda(t)\partial_3L(t,x(t),\LD x(t),z(t))\right)=0$$
on $[a,b]$.
\end{theorem}

\begin{proof} Let $x+\epsilon h$ be a variation of $x$, with $|\epsilon| \ll 1$ and $h\in C^1[a,b]$ such that $h(a)=0=h(b)$. The variation of $z$ is given by
$$\theta (t)=\frac{d}{d\epsilon} \left.z[t,x+\epsilon h]\right|_{\epsilon=0}.$$
Inserting the variations into the differential equation \eqref{Herglotz}, we obtain
$$\frac{d}{dt}z[t,x+\epsilon h]=L(t,x(t)+\epsilon h(t),\LD x(t)+\epsilon\LD h(t),z[t,x+\epsilon h]).$$
Since
$$\begin{array}{ll}
\theta'(t)&=\displaystyle \left.\frac{d}{d\epsilon} \frac{d}{dt}z[t,x+\epsilon h]\right|_{\epsilon=0}\\
&=\displaystyle\frac{d}{d\epsilon }
\left. L(t,x(t)+\epsilon h(t),\LD x(t)+\epsilon\LD h(t),z[t,x+\epsilon h])\right|_{\epsilon=0}\\
&=\displaystyle \partial _2L (t,x(t),\LD x(t),z(t))h(t)+\partial _3L (t,x(t),\LD x(t),z(t))\LD h(t)\\
&\quad\displaystyle+\partial _4L (t,x(t),\LD x(t),z(t))\theta(t),\\
\end{array}$$
we obtain a linear differential equation whose solution is given by
$$\theta(t)\lambda(t)-\theta(a)$$
$$=\int_a^t \lambda(\t)\left[ \partial _2L (\t,x(\t),\LD x(\t),z(\t))h(\t)+\partial _3L (\t,x(\t),\LD x(\t),z(\t))\LD h(t) \right]d\tau,$$
and integrating by parts, we get
\begin{equation}\label{Herglotz2}\theta(t)\lambda(t)-\theta(a)\end{equation}
$$=\int_a^t \left[\lambda(\t) \partial _2L (\t,x(\t),\LD x(\t),z(\t))-\RD\left(\lambda(\t)\partial _3L (\t,x(\t),\LD x(\t),z(\t))\right)\right]h(\t)\,d\tau.$$
Since $z(a)$ is fixed, we get $\theta(a)=0$. Also, since $z(b)$ is minimum, then $\theta (b)=0$. Replacing $t=b$ in Eq. \eqref{Herglotz2}, and by the arbitrariness of $h$ in $(a,b)$, we obtain that
$$\lambda(t)\partial_2L(t,x(t),\LD x(t),z(t))-\RD\left(\lambda(t)\partial_3L(t,x(t),\LD x(t)z(t))\right)=0$$
must hold for all $t\in[a,b]$.
\end{proof}


\section*{Acknowledgments}

I would like to thank Tatiana Odzijewicz,  for a careful and thoughtful reading of the manuscript.
This work was supported by Portuguese funds through the CIDMA - Center for Research and Development in Mathematics and Applications,
and the Portuguese Foundation for Science and Technology (FCT-Funda\c{c}\~ao para a Ci\^encia e a Tecnologia), within project UID/MAT/04106/2013.


\end{document}